\documentclass[a4paper,10pt]{article}
\usepackage{}
\usepackage{pifont}
\usepackage{mathrsfs}
\usepackage{tipa}
            \newenvironment{proof}{\emph{Proof:}}{\hspace{\stretch{1}}\rule{1ex}{1ex}} 
            \newenvironment{keywords}{\emph{Keywords.}}{\hspace{\stretch{1}}}          
            \newtheorem{definition}{Definition}                                        
            \newtheorem{lemma}{Lemma}                                                  
            \newtheorem{remark}{Remark}                                                
            \newtheorem{theorem}{Theorem}                                              
            \def\tiltlefyle{\tiltlefyleA}                                              

\usepackage[centertags]{amsmath}
\usepackage{graphicx}
\usepackage{cite}
\usepackage{amsfonts}
\usepackage{amssymb}
\usepackage{color}

\def\tiltlefyleA{\title{\LARGE\bf{\titlec}\thanks{\fadationc}}     
\author{\authorc \thanks{\adressc. Email: \emailc}         
, \authorcc  \thanks{\adresscc. Email: \emailcc}      
}                                                                  
\date{}\maketitle                                                  
\begin{abstract}\abstractc\end{abstract}                           
\begin{keywords}\keywordsc\end{keywords}
\bibliographystyle{plain}}                                         

\def\titlec{Improved finite-time stability and instability theorems for stochastic nonlinear systems}

\def\fadationc{This paper was not presented at any IFAC
meeting.
}

\def\authorc{Weihai~Zhang}                                                                                                        
\def\adressc{College of Electrical Engineering and Automation, Shandong
University of Science and Technology, Qingdao,  Shandong Province, 266590,
P. R. China. Corresponding author}        
\def\emailc{w\_hzhang@163.com}
\def\authorcc{Liqiang~Yao}                                                                                                        
\def\adresscc{School of Mathematics and Information Science, Yantai University, Yantai,
Shandong Province, 264005, P. R. China}                            
\def\emailcc{liqiangyaoyt@163.com}
\def\abstractc{This paper studies finite-time stability and instability theorems in
probability sense for stochastic nonlinear systems. Firstly, a new
sufficient condition is proposed to guarantee that the considered
system has a global solution. Secondly, we propose improved
finite-time stability and instability criteria that relax the
constraints on $\mathcal {L}V$ (the infinitesimal operator of
Lyapunov function $V$) by the uniformly asymptotically stable
function(UASF). The improved finite-time stability theorems allow
$\mathcal {L}V$ to be indefinite (negative  or positive) rather than
just only allow $\mathcal {L}V$ to be negative. Most existing
finite-time stability and instability results can be viewed as
special cases of the obtained theorems. Finally, some simulation
examples verify the validity of the theoretical  results.
}

\def\keywordsc{Stochastic nonlinear systems; finite-time stability;
finite-time instability; uniformly asymptotically stable function(UASF).}

\begin{document}
\tiltlefyle
\section{Introduction}\label{S1}
Stability plays a central role in systems theory and engineering
applications, and is always the most fundamental consideration in
system analysis and synthesis. The two most commonly used concepts
in  stability theory  are asymptotic stability
[\ref{Khalil1992}-\ref{xuelingrong}] and finite-time stability
[\ref{Haimo1986}-\ref{Nersesov2008}]. The asymptotic stability
describes the asymptotic behavior of the system state when time
tends to infinity, however, the finite-time stability illustrates
the transient performance of the system state trajectory within a
finite time interval. For many engineering problems, one often pays
more attention to  finite time stability instead of   asymptotic
stability; see, e.g.,  the tracking control of robotic manipulators
[\ref{Yu2005}],  the position and orientation control of underwater
vehicles [\ref{Banavar2006}]. Specially, the finite-time stable
system usually shows better robustness and faster convergence rate.
So finite-time stability and its related control problems have
aroused great interest  during the past two decades; see
[\ref{Wang2013}-\ref{Ming2021}].

Many practical systems are subject to stochastic perturbations  such
as  environmental noise. Actual systems in  random environments are
often modeled as stochastic systems, which can be studied by
stochastic analysis method. It is natural that the  finite-time
stability and stabilization of stochastic systems have attracted  a
great deal of attention.
For example, [\ref{Yin2011}] presented  the definition of
finite-time stability and established finite-time stability and
instability criteria in probability sense under the hypothesis that
the considered system has a unique strong solution. Generally
speaking, it is difficult to ensure that there exists a unique
strong solution to stochastic nonlinear systems without Lipschitz
condition. So researchers have to study  finite-time stability and
stabilization problems for  stochastic nonlinear systems under the
framework of strong solution or weak solution. [\ref{Yu2019}]
generalized the finite-time stability criteria of  [\ref{Yin2011}]
and relaxed the constraint conditions  on  $\mathcal {L}V$ (the
infinitesimal generator of Lyapunov function) in  strong solution or
weak solution sense. Recently, [\ref{Yu2021}] has presented new
finite-time stability results by multiple Lyapunov functions for
stochastic nonlinear time-varying systems. In addition,
[\ref{Wang2015}],  [\ref{xiexuejun1}], [\ref{Min2017}] and
[\ref{Zhao2018}]  studied feedback  stabilization problems for
several classes of stochastic nonlinear strict-feedback systems in
finite time. [\ref{Huang2016}] and [\ref{Song2019}] discussed the
finite-time stabilization for two classes of stochastic nonlinear
high-order systems with switching, respectively.

In this paper, we focus on the existence of the system solution and
its finite-time stability  and instability. A new sufficient
condition is given to guarantee that the considered system has a
global solutio,   which  is weaker than the  sufficient condition
appeared  in Lemma 1 of [\ref{Yu2021}]. It can be found that, in
existing literature,   most finite-time stability theorems are based
on  $\mathcal {L}V$ to be negative  (e.g., [\ref{Yin2011},
\ref{Zhao2018}]) or non-positive (e.g., [\ref{Yu2019}]), where  $V$
is a Lyapunov  function. Similarly, most finite-time instability
theorems of stochastic nonlinear systems also  require $\mathcal
{L}V$ to be negative (e.g., [\ref{Yin2011}, \ref{Yu2019}]). In order
to relax the constraints on  Lyapunov  function, this paper proposes
improved finite-time stability and instability theorems. By UASF
 proposed in [\ref{Zhou2018}, \ref{Zhou2016}], we establish
some improved finite-time stability/instability theorems in which
$\mathcal {L}V$ can be  negative definite  or  positive definite
rather than just only negative definite. On the other hand, the
obtained results in this paper  contain or cover  many existing
sufficient conditions about finite-time stability and instability
criteria as  special cases.

The rest of this paper is organized as follows. Mathematical
preliminaries  are given in Section 2. Section 3 and  Section 4
discuss  finite-time stability and instability of stochastic
nonlinear systems, respectively. In Section 5, an example is given
to show the application of our improved finite-time stability
theorems. Section 6 concludes  this paper  with some remarks.

{\bfseries Notations:} The set of all natural numbers is denoted by
$\mathcal {N}$. $\mathcal {R}_+$ denotes the family of all
nonnegative real numbers and $\mathcal{R}^r$ denotes the real
$r$-dimensional space.  $|x|$ is the Euclidean norm of a vector $x$.
The Frobenius norm of a  real  matrix $X$ is
$||X||=[\text{Tr}(X^TX)]^{1/2}$. $C^{1,2}([t_0, \infty) \times
\mathcal{R}^r; \mathcal{R}_+)$ stands for the set of all nonnegative
functions $w(t, x)$ that are $C^1$ in $t$ and  $C^2$ in $x$.
$C^{1,2}_0([t_0, \infty) \times \mathcal{R}^r; \mathcal{R}_+)$
denotes the set of all nonnegative functions $w(t, x)\in
C^{1,2}([t_0, \infty) \times \mathcal{R}^r; \mathcal{R}_+)$ except
possibly at the point $x=0$. $\mathcal{K}$ represents the set of all
strictly increasing, continuous functions $\gamma(t)$  with
$\gamma(0)=0$. $\mathcal{K}_\infty$ denotes the set of all unbounded
functions $\gamma(t)$ with  $\gamma(t)\in\mathcal{K}$. $\mathcal{KL}$
denotes the set of all functions
$\beta(s,t)\in\mathcal{K}$ with $t$ being fixed and $\lim_{t
\rightarrow \infty}\beta(s,t)=0$ with $s$ being fixed.


\section{Preliminaries}\label{S2}
Consider  the stochastic nonlinear It\^o system
\begin{equation}\label{eq21}
dx(t)=f(t, x(t))dt+g(t, x(t))dW(t),
\end{equation}
in which system state $x(t)\in \mathcal{R}^{r}$ and $x(t_0)=x_0$,
$W(t)\in \mathcal{R}^{d}$,  defined on the filtered probability
space $(\Omega, \mathcal{F}, \{\mathcal {F}_t\}_{t\geq t_0}, P)$,
is a standard Wiener process. The  continuous functions $f: [t_0,
\infty) \times \mathcal{R}^r \rightarrow \mathcal{R}^r $ and $g:
[t_0, \infty) \times \mathcal{R}^r \rightarrow \mathcal{R}^{r\times
d}$ satisfy $f(t, 0)=0$ and $g(t, 0)=0$ with $t\in [t_0, \infty)$.

According to Lemma 3.2 in [Chapter 4, \ref{Mao2007}], the trivial
solution of system (\ref{eq21}) is impossible to reach  the origin
in a finite time, if  $f(t, x)$ and  $g(t, x)$ satisfy   local
Lipschitz condition in $x$. Thus some conditions should be imposed
on system (\ref{eq21}) such that system (\ref{eq21}) has a solution,
which may arrive at the origin in a finite time. The following Lemma
\ref{lemma21}  ensures that system (\ref{eq21}) has a unique strong
solution and it can be viewed as a special case of Theorem 170 in
[\ref{Situ2005}].

\begin{lemma}\label{lemma21}
For every $N=1,2,\cdots,$ and any $T\in [t_0, \infty)$,
if system (\ref{eq21}) satisfies that\\
\textbf{H1}: $|f(t, x)|\leq (1+|x|)h(t)$, $||g(t, x)||^2\leq (1+|x|^2)h(t)$ and\\
 \textbf{H2}: $2(x_1-x_2, f(t, x_1)-f(t, x_2))
 +||g(t, x_1)-g(t, x_2)||^2
\leq \psi_T^N(|x_1-x_2|^2)h_T^N(t)$, $|x_1|\vee |x_2|\leq N$\\
for any $t\in[t_0, T]$, where the nonnegative functions $h_T^N(t)$
and $h(t)$ satisfy $\int_{t_0}^Th_T^N(t)dt<\infty$ and
$\int_{t_0}^Th(t)dt<\infty$. Moreover, for any $v\geq 0$,
$\psi_T^N(v)\geq 0$ is a continuous, strictly increasing, non-random
and concave function satisfying $\int_{0+}dv/\psi_T^N(v)=\infty$.
Then, system (\ref{eq21}) has a unique solution for any given
initial value $x_{0}\in \mathcal{R}^r$.
\end{lemma}

\begin{remark}\label{rem20}
Note that  \textbf{H2} in  Lemma \ref{lemma21} is very restrictive
for system (\ref{eq21}), which  limits the application  scope  of
finite-time stability theorems. In Lemma \ref{lemma21},
\textbf{H2} is used to show  the uniqueness of solution to system
(\ref{eq21}).
\end{remark}

The following Lemma \ref{lemma22} and Lemma \ref{lemma23} give some sufficient
conditions, which ensure that system (\ref{eq21}) has a  continuous
strong solution.

\begin{lemma}\label{lemma22}
If system (\ref{eq21}) satisfies that
\textbf{H3}: $|f( t, x)|^2+||g(t, x)||^2\leq H(1+|x|^2)$ where
$H>0$ is a constant, then system (\ref{eq21}) has a continuous solution
with probability 1.
\end{lemma}

\begin{lemma}\label{lemma23}
For system (\ref{eq21}), assume that  $f(t, x)$ and $g(t, x)$ are
locally bounded in $x$  and are  uniformly  bounded  in $t$. If
there exist a function $U(t,x)\in C^{1,2}([t_0, \infty) \times
\mathcal{R}^r; \mathcal{R}_+)$, a  $\mathcal{K}_\infty$ function
${\gamma}(\cdot)$, a continuous function $l(t)$ and a constant
$d_U\geq 0$ such that $\gamma(|x|)\leq U(t,x)$ and $\mathcal
{L}U(t,x)\leq l(t)U(t,x)+d_U$ hold for any $x\in \mathcal{R}^r$,
then system (\ref{eq21}) has a continuous solution with probability
1.
\end{lemma}

\begin{proof}
For each integer $k \geq 1$, let
\begin{align*}
f_k(t, x)=
\begin{cases}
f(t, \frac{kx}{|x|}) &\textmd{if}~~|x|\geq k \\
f(t, x) &\textmd{if}~~|x|< k \\
\end{cases},~~~~
g_k(t, x)=
\begin{cases}
g(t, \frac{kx}{|x|}) &\textmd{if}~~|x|\geq k \\
g(t, x) &\textmd{if}~~|x|< k \\
\end{cases},
\end{align*}
together with $f(t, x)$ and $g(t, x)$ being locally bounded in $x$
and  uniformly  bounded in $t$.  Then $|f_k(t, x)|\leq \sup_{|y|\leq
k, t\geq t_0}|f(t, y)|\triangleq c_{1k}<\infty,~~ |g_k(t, x)|\leq
\sup_{|y|\leq k, t\geq t_0}|g(t,y)|\triangleq c_{2k}<\infty.$
According to Theorem 5.2 in [\ref{Skorokhod1965}], we know that
stochastic nonlinear system
\begin{equation}\label{eq22}
dx_k(t)=f_k(t, x_k(t))dt+g_k(t, x_k(t))dW(t)
\end{equation}
has a continuous solution $x_k(t)$ on $[t_0, \infty)$.
Further, define $\tau_k=\inf\{t\geq t_0: |x_k(t)|\geq k\}$, then
(\ref{eq22}) can be rewritten as
\begin{equation}\label{eq23}
dx_k(t)=f(t, x_k(t))dt+g(t, x_k(t))dW(t),
\end{equation}
where $t\in[t_0, \tau_k)$.  For any $T\in [t_0, \infty)$,
let $\tilde{U}(t,x)=U(t,x)\exp\{-\int_{t_0}^tl(s)ds\}$,
then we have
\begin{align*}
\mathcal {L}\tilde{U}(t,x_k(t))
= (\mathcal {L}{U}(t,x_k(t))-l(t){U}(t,x_k(t)))e^{-\int_{t_0}^tl(v)dv}
\leq d_Ue^{-\int_{t_0}^tl(v)dv}
\end{align*}
from which we can further obtain that
\begin{align}\label{eq24}
E[U(\tau_k\wedge T, x(\tau_k\wedge T))]
\leq &E[{U}(t_0,x_0)e^{\int_{t_0}^{\tau_k\wedge T}l(v)dv}]
+d_UE[\int_{t_0}^{\tau_k\wedge T}e^{{\int_{t_0}^{\tau_k\wedge T}l(v)dv}-\int_{t_0}^sl(v)dv}ds]\cr
\leq &{U}(t_0,x_0)e^{\int_{t_0}^{T}|l(v)|dv}
+d_U\int_{t_0}^{ T}e^{\int_{s}^{ T}|l(v)|dv}ds\cr
\leq &{U}(t_0,x_0)e^{M_T(T-t_0)}
+d_UTe^{M_T(T-t_0)},
\end{align}
where $M_T=\max_{t_0\leq t\leq T}\{|l(t)|\}$. Note that
$\inf_{|x|=k, t\geq t_0} U(t, x)P\{\tau_k\leq T\}\leq E[U(\tau_k\wedge T, x(\tau_k\wedge T))], $
then (\ref{eq24}) can be turned into
\begin{equation}\label{eq25}
P\{\tau_k\leq T\}\leq \frac{{U}(t_0,x_0)+d_UT}{\inf_{|x|=k, t\geq t_0} U(t, x)}
e^{M_T(T-t_0)}.
\end{equation}

Let $k\rightarrow \infty$ in (\ref{eq25}), then we get that
$\lim_{k\rightarrow \infty}P\{\tau_k\leq T\}=0$ for any $T\in [t_0,
\infty)$, where
$$
\lim_{k\rightarrow \infty}\inf_{|x|=k, t\geq t_0}
U(t, x)=\infty
$$
is used. This means that $\lim_{k\rightarrow \infty}\tau_k=\infty$
a.s.. Hence, for any $t\in[t_0, \tau_k)$, we define $x(t)=x_k(t)$ in
(\ref{eq23}), then $x(t)$ is the global solution of  system
(\ref{eq21}).
\end{proof}

\begin{remark}\label{rem201}
(i) Lemma \ref{lemma22} comes from Theorem 5.2 in
[\ref{Skorokhod1965}]. Lemma \ref{lemma23} is an improved version of Lemma
1 in [\ref{Yu2021}]. In Lemma 1 of [\ref{Yu2021}],  $l(t)$ is a
nonnegative function with $\int_{t_0}^\infty |l(t)|dt<\infty$ and
$d_U=0$. However, Lemma \ref{lemma23} removes these strict
constraints and allow $l(t)$ to be a continuous function and
$d_U\geq 0$.
\end{remark}

\begin{definition}\label{def21} ([\ref{Yin2011}])
If system (\ref{eq21}) has a   solution $x(t)$ and satisfies that \\
(i) For any $x_0\in\mathcal{R}^r\backslash\{0\}$,
the stochastic settling time $\rho_{x_{0}}=\inf\{t\geq t_0: x(t)=0\}$
is a finite time a.s., i.e, $P\{\rho_{x_{0}}<\infty\}=1$.\\
(ii) There is a positive constant $\delta(\varepsilon, R)$ such that
$P\{|x(t)|<R, ~~ t\geq t_0\}\geq 1-\varepsilon$ with
$|x_{0}|<\delta(\varepsilon, R)$  for any
$0<\varepsilon<1$ and $R>0$.\\
Then,  the trivial solution of system (\ref{eq21}) is called finite-time stable in probability.
\end{definition}

\begin{definition}\label{def22} ([\ref{Yu2021}])
In Definition \ref{def21}, if (i) or (ii) cannot be satisfied, then
the trivial solution of system (\ref{eq21}) is finite-time instable
in probability.
\end{definition}

\begin{definition}\label{def25} ([\ref{Zhou2018}])
For system $\dot{z}(t)=\mu(t)z(t)$ with $t\in[t_0,\infty)$, if there exists a function
$\beta(\cdot, \cdot)\in \mathcal {K}\mathcal {L}$ such that
$|z(t)|\leq \beta(|z(t_0)|,t-t_0)$, then this system is
globally uniformly asymptotically stable and $\mu(t)$ is called a UASF.
\end{definition}

\begin{remark}\label{rem21}
According to [\ref{Zhou2016}], $\mu(t)$ is a UASF if and only if
there exist constants $c_\mu>0$ and $d_\mu \geq 0$ such that
$\int_{t_0}^t\mu(s)ds\leq d_\mu-c_\mu(t-t_0)$.
\end{remark}

For any given
$V(t, x)\in C^{1,2} ([t_0,\infty) \times \mathcal{R}^r ; \mathcal{R}_+)$
associated with system (\ref{eq21}), the infinitesimal operator $\mathcal {L}$
is defined as
$\mathcal{L}V(t, x)= V_t(t, x)+V_x(t, x)f(t, x)
+\frac{1}{2}\text{Tr}\{g^T(t, x)V_{xx}(t, x)g(t, x)\}$.

The main goal  of this paper is to present  the improved finite-time
stability and instability theorems under the condition that system
(\ref{eq21}) has a solution, which is in  weak solution  or strong
solution  sense. By stochastic analysis technique and UASF, the
established finite-time stability and instability theorems  can
break through some strict constraints imposed on  the existing
finite-time stability and instability results.

\section{Finite-time stability theorems}\label{S3}
In this section, our purpose is to build some improved finite-time
stability theorems of stochastic nonlinear systems. Compared with
some  existing relative  results, the improved finite-time stability
theorems weaken the  condition on $\mathcal{L}V$ and permit
$\mathcal{L}V$ to be indefinite. For this end, we need the following
Lemma \ref{lemma31}, which plays an important role
in the proof of Theorem \ref{theorem31}.

\begin{lemma}\label{lemma31}
For system (\ref{eq21}), if there exist  a function
$V^*(t,x)\in C^{1,2}_0([t_0, \infty) \times U_k; \mathcal{R}_+)$
and a  UASF $\mu^*(t)$ such that for any $\varepsilon\in (0,k)$,
\begin{align}\label{lemma31eq1}
\mathcal {L}V^*(t,x)\leq \mu^*(t),\forall x\in U_{k,\varepsilon},
\end{align}
where $U_k=\{x\in \mathcal{R}^r: |x|<k\}$ and
$U_{k,\varepsilon}=\{x\in \mathcal{R}^r: \varepsilon<|x|<k\}$. Then,
the trivial solution of system (\ref{eq21}) with $x_{0}\in
U_{k,\varepsilon}$ first reaches the boundary of $U_{k,\varepsilon}$
in  finite time a.s..
\end{lemma}

\begin{proof}
Let $\varpi(t)=-\int_{t_0}^{t}\mu^*(s)ds$, then
$\varpi(t)\geq c_{\mu^*}(t-t_0)-d_{\mu^*}$ and
$\lim_{t\rightarrow\infty}\varpi(t)=+\infty$ by Remark \ref{rem21}.
Moreover,  if we let $T_0=\max\{T\geq t_0: \int_{t_0}^T\mu^*(t)dt=0\}$,
then $\varpi(T_0)=0$ with $T_0<\infty$ and $\varpi(t)>0$ as $t>T_0$.

Let $\varrho^*_{k,\varepsilon}=\inf\{t\geq t_0: |x(t)|\notin U_{k, \varepsilon}\}$,
then we need to prove that $P\{\varrho^*_{k,\varepsilon}< \infty\}=1$, i.e.,
$P\{\varrho^*_{k,\varepsilon}=\bar{\varrho}_{k,\varepsilon}I_{\{\varrho^*_{k,\varepsilon}\leq T_0\}}
+\varrho_{k,\varepsilon}I_{\{\varrho^*_{k,\varepsilon}> T_0\} }< \infty\}=1$ with
$\bar{\varrho}_{k,\varepsilon}=\inf\{t_0 \leq t\leq T_0: |x(t)|\notin U_{k, \varepsilon}\}$ and
$\varrho_{k,\varepsilon}=\inf\{t> T_0: |x(t)|\notin U_{k, \varepsilon}\}$.

Firstly, we prove the case that
$P\{\varrho^*_{k,\varepsilon}=\varrho_{k,\varepsilon}<\infty \}=1$.
From Dynkin's formula and (\ref{lemma31eq1}), we can get
that for any $x\in U_{k,\varepsilon}$,
\begin{align}\label{lemma31eq2}
EV^*(\varrho_{k,\varepsilon}\wedge t, x(\varrho_{k,\varepsilon}\wedge t))
=&EV^*(T_0, x(T_0))+E\int_{T_0}^{\varrho_{k,\varepsilon}\wedge t}\mathcal {L}V^*(s, x(s))ds\cr
\leq &EV^*(T_0, x(T_0))+E\int_{T_0}^{\varrho_{k,\varepsilon}\wedge t}\mu^*(s)ds,
\end{align}
where $t>T_0$. Define $\varpi_1(t)=-\int_{T_0}^{t}\mu^*(s)ds$, then
$\varpi_1(t)=\varpi(t)>0$ as $t> T_0$ and
$\lim_{t\rightarrow\infty}\varpi_1(t)=\lim_{t\rightarrow\infty}\varpi(t)=+\infty$.
Further, (\ref{lemma31eq2}) can be turned into
$EV^*(\varrho_{k,\varepsilon}\wedge t, x(\varrho_{k,\varepsilon}\wedge t))
\leq EV^*(T_0,x(T_0))-E\varpi_1(\varrho_{k,\varepsilon}\wedge t),$
which means that
\begin{align}\label{lemma31eq4}
E\varpi_1(\varrho_{k,\varepsilon}\wedge t)\leq EV^*(T_0,x(T_0)), ~~ t> T_0.
\end{align}
Note that for any $t>T_0$,
\begin{align}\label{lemma31eq5}
\varpi_1(t)P\{\varrho_{k,\varepsilon}\geq t\}
=\int_{\{\varrho_{k,\varepsilon}\geq t\}}\varpi_1(t)dP
\leq E\varpi_1(\varrho_{k,\varepsilon}\wedge t).
\end{align}
Substituting (\ref{lemma31eq5}) into (\ref{lemma31eq4}) arrives at
\begin{equation}\label{lemma31eq6}
P\{\varrho_{k,\varepsilon}\geq t\}\leq\frac{1}{\varpi_1(t)}EV^*(T_0,x(T_0)),
\end{equation}
where $t> T_0$. In (\ref{lemma31eq6}), let $t\rightarrow\infty$,
then $P\{\varrho_{k,\varepsilon}< \infty\}=1$, that is to say,
$P\{\varrho^*_{k,\varepsilon}=\varrho_{k,\varepsilon}<\infty\}=1$
can be obtained.

Secondly, we  prove another case that
$P\{\varrho^*_{k,\varepsilon}=\bar{\varrho}_{k,\varepsilon}<\infty \}=1$.
In fact, it is natural because
$P\{\varrho^*_{k,\varepsilon}=\bar{\varrho}_{k,\varepsilon}\leq T_0\}=1$.

Hence, we have
$P\{\varrho^*_{k,\varepsilon}=\bar{\varrho}_{k,\varepsilon}I_{\{\varrho^*_{k,\varepsilon}\leq
T_0\}} +\varrho_{k,\varepsilon}I_{\{\varrho^*_{k,\varepsilon}> T_0\}
}< \infty\}=1$, i.e., $P\{\varrho^*_{k,\varepsilon}< \infty\}=1$.
This means that the solution of system (\ref{eq21}) with $x_{0}\in
U_{k,\varepsilon}$ firstly arrives at the boundary of
$U_{k,\varepsilon}$ in finite time almost surely.
\end{proof}

\begin{theorem}\label{theorem31}
Suppose that there exists  a solution to system (\ref{eq21}).
If there are
functions $\underline{\gamma}(\cdot)\in\mathcal{K}_\infty$,
$\bar{\gamma}(\cdot)\in\mathcal{K}_\infty$,
$V(t,x)\in C^{1,2}([t_0, \infty) \times \mathcal{R}^r; \mathcal{R}_+)$
and a UASF $\mu(t)$, and a constant $\kappa \in (0,1)$ such that
\begin{gather}
\label{theorem31eq1}
\underline{\gamma}(|x|) \leq V(t, x)\leq \bar{\gamma}(|x|),\\
\label{theorem31eq2} \mathcal {L}V(t,x))\leq \mu(t)[V(t,x)]^\kappa,
\end{gather}
then  the solution of  system  (\ref{eq21}) is finite-time stable in probability.
\end{theorem}

\begin{proof}
If system initial value $x_{0}=0$, then $x(t)\equiv0$ is
the solution of system  (\ref{eq21}),
where $f(0,t)\equiv0$ and $g(0,t)\equiv0$ are considered.
In the following, we only discuss the case that
$x(t_0)=x_0\in\mathcal{R}^r\backslash\{0\}$.

Defining ${W}(V)=\int_{0}^V s^{-\kappa}ds$ with $V\in(0, +\infty)$
along with system  (\ref{eq21}) and using Lemma 3.1 in
[\ref{Zhao2015}], then
\begin{align}\label{theorem31eq01}
\mathcal {L} {W}(V(t, x))
=\frac{\mathcal {L}V}{V^{\kappa}}
-\frac{\kappa }{2}\frac{\textmd{Tr}\{(V_xg)^TV_xg\}}{V^{\kappa+1}}
\leq V^{-\kappa}\mathcal {L}V
\leq \mu(t), ~~x\in\mathcal{R}^r\backslash\{0\}.
\end{align}
Let $\rho_k=\varrho_{k,\frac{1}{k}}=\inf\{t\geq t_0: |x(t)|\notin U_{k, \frac{1}{k}}\}$
with $U_{k, \frac{1}{k}}=\{x\in \mathcal{R}^r: \frac{1}{k}<|x|<k\}$ and
$k\in \{2, 3, \cdots\}$, then by (\ref{theorem31eq01}),  we must have
\begin{align}\label{theorem31eq0101}
\mathcal {L} {W}(V(t, x))
\leq \mu(t), ~~x\in U_{k, \frac{1}{k}}.
\end{align}
By Dynkin's formula, it follows from (\ref{theorem31eq0101}) that
\begin{align}\label{theorem31eq02}
- {W}(V(t_0,x_0))
\leq E {W}(V(t\wedge \rho_k, x(t\wedge \rho_k)))
- {W}(V(t_0,x_0))
\leq E\int_{t_0}^{t\wedge \rho_k}\mu(s)ds.
\end{align}
Because $\mu(t)$ is a UASF, there exist constants
$c_\mu>0$ and $d_\mu \geq 0$ such that
$\int_{t_0}^t\mu(s)ds\leq d_\mu-c_\mu(t-t_0)$.
Hence, (\ref{theorem31eq02}) can  be turned into
\begin{align*}%
-\frac{1}{1-\kappa}V^{1-\kappa}(t_0, x_0)=&-\int_{0}^{V(t_0, x_0)} s^{-\kappa}ds
\leq -{c_\mu}E(t\wedge \rho_k)+{c_\mu}t_0+{d_\mu},
\end{align*}
that is to say,
\begin{align}\label{theorem31eq03}
E(t\wedge \rho_k)\leq t_0+\frac{d_\mu}{c_\mu}
+\frac{\bar{\gamma}^{1-\kappa}(|x_0|)}{c_\mu(1-\kappa)},
\end{align}
where (\ref{theorem31eq1}) is considered.

Let $ {W}=V^*$ and $\mu(t)=\mu^*(t)$ in (\ref{theorem31eq01}), then
we have $\rho_k \rightarrow \rho_\infty=\rho_{x_0}$ a.s. as $k
\rightarrow \infty$ and $ P\{\rho_{x_0}<\infty\}=1$ by Lemma
\ref{lemma31} and Definition \ref{def21}. Further, let $t=k$ and
$k\rightarrow \infty$, then $k\wedge\rho_k \rightarrow \rho_{x_0}$
a.s.. Accordingly, (\ref{theorem31eq03}) can be changed into
\begin{equation}\label{theorem31eq0404}
E(\rho_{x_0})\leq t_0+\frac{d_\mu}{c_\mu}
+\frac{\bar{\gamma}^{1-\kappa}(|x_0|)}{c_\mu(1-\kappa)}.
\end{equation}
This implies that the trivial solution of system (\ref{eq21})
is finite-time attractive in probability.

Applying Dynkin's formula for (\ref{theorem31eq01}),
together with (\ref{theorem31eq1}), leads to
\begin{align}\label{theorem31eq04}
E[\underline{\gamma}^{1-\kappa}(|x(t)|)I_{\{t\in[t_0, \rho_{x_0})\}}]
\leq E[V^{1-\kappa}(t, x(t))I_{\{t\in[t_0, \rho_{x_0})\}}]
\leq \bar{\gamma}^{1-\kappa}(|x_0|)+d_\mu(1-\kappa).
\end{align}
Meanwhile,
\begin{equation}\label{theorem31eq05}
E[\underline{\gamma}^{1-\kappa}(|x(t)|)I_{\{t\in[\rho_{x_0}, \infty)\}}]\equiv 0.
\end{equation}
Hence, we have
\begin{align}\label{theorem31eqs07}
E[\underline{\gamma}^{1-\kappa}(|x(t)|)]
=&E[\underline{\gamma}^{1-\kappa}(|x(t)|)I_{\{t\in[t_0, \rho_{x_0})\}}
+\underline{\gamma}^{1-\kappa}(|x(t)|)I_{\{t\in[\rho_{x_0}, \infty)\}}]\cr
=&E[\underline{\gamma}^{1-\kappa}(|x(t)|)I_{\{t\in[t_0, \rho_{x_0})\}}]
+E[\underline{\gamma}^{1-\kappa}(|x(t)|)I_{\{t\in[\rho_{x_0}, \infty)\}}]\cr
\leq &\bar{\gamma}^{1-\kappa}(|x_0|)+d_\mu(1-\kappa),~~t\geq t_0
\end{align}
from (\ref{theorem31eq04}) and  (\ref{theorem31eq05}).

For any $\varepsilon\in(0,1)$,
let $\bar{R}>\varepsilon^{-1}\bar{\gamma}^{1-\kappa}(|x_0|)+\varepsilon^{-1}d_\mu(1-\kappa)$,
then from (\ref{theorem31eqs07}) and Chebyshev's inequality, we can get that
\begin{align*}
P\{\underline{\gamma}^{1-\kappa}(|x(t)|)\geq \bar{R}, ~~t\geq t_0\}
\leq E[\underline{\gamma}^{1-\kappa}(|x(t)|)]\bar{R}^{-1}
<\varepsilon,
\end{align*}
which means that
\begin{equation}\label{theorem31eqs08}
P\{|x(t)|< R,~~t\geq t_0\}\geq 1-\varepsilon,
\end{equation}
where $R=\underline{\gamma}^{-1}(\bar{R}^{\frac{1}{1-\kappa}})$,
$|x_0|<\delta(\varepsilon, R)$ and $\delta(\varepsilon, R)
=\bar{\gamma}^{-1}([\varepsilon\underline{\gamma}^{1-\kappa}(R)
-d_\mu(1-\kappa)]^{\frac{1}{1-\kappa}})$. This signifies that the
trivial solution of system (\ref{eq21}) is stable in probability.
\end{proof}
\begin{theorem}\label{theorem32}
Suppose that there is a  unique solution to system (\ref{eq21}).
If there are
functions $\underline{\gamma}\in\mathcal{K}_\infty$, $\bar{\gamma}\in\mathcal{K}_\infty$,
a UASF $\mu(t)$ and $V(t,x)\in C^{1,2}([t_0, \infty) \times \mathcal{R}^r; \mathcal{R}_+)$
such that
\begin{gather}
\label{theorem32eq1}
\underline{\gamma}(|x|) \leq V(t, x)\leq \bar{\gamma}(|x|), \\
\label{theorem32eq2} \mathcal {L}V(t,x)\leq \mu(t),
\end{gather}
then  the trivial solution of  system  (\ref{eq21})
is finite-time stable in probability.
\end{theorem}

\begin{proof}
Similar to Theorem \ref{theorem31},  we also only consider the case that
$x(t_0)=x_0\in\mathcal{R}^r\backslash\{0\}$ in the following.

Applying Dynkin's formula for (\ref{theorem32eq2})
and the property of  the UASF $\mu(t)$, we have
\begin{align}\label{theorem32eqa01}
EV(t\wedge \sigma_n \wedge \rho_{x_{0}},
x(t\wedge \sigma_n \wedge \rho_{x_{0}}))
\leq V(t_0, x_0)-c_\mu E(t\wedge \sigma_n \wedge \rho_{x_{0}}-t_0)+d_\mu,
\end{align}
where $\sigma_n=\inf\{t\geq t_0:|x(t)|>n\}$ and $\rho_{x_{0}}$
represents the stochastic settling time. Note that
$\sigma_n\rightarrow\infty$ as $n\rightarrow\infty$. So by  Fatou's
lemma, (\ref{theorem32eq1}) and (\ref{theorem32eqa01}), we have
\begin{align*}
c_\mu E(t \wedge \rho_{x_0})
\leq EV(t\wedge \rho_{x_0},x(t\wedge \rho_{x_0}))
+c_\mu E(t \wedge \rho_{x_0})
\leq \bar{\gamma}(|x_0|)+c_\mu t_0+d_\mu,
\end{align*}
which indicates that
\begin{align}\label{theorem32eqa02}
E(t \wedge \rho_{x_0})
\leq \frac{\bar{\gamma}(|x_0|)}{c_\mu}+\frac{d_\mu}{c_\mu}+t_0.
\end{align}
Applying  Fatou's lemma for  (\ref{theorem32eqa02}), we have
\begin{align}\label{theorem32eqa03}
E(\rho_{x_0})
\leq \frac{\bar{\gamma}(|x_0|)}{c_\mu}+\frac{d_\mu}{c_\mu}+t_0,
\end{align}
which signifies that  $P\{\rho_{x_0}<\infty\}=1$.

The proof of stability in probability is similar to that of
Theorem \ref{theorem31} and is thus omitted here.
\end{proof}

\begin{remark}\label{rem33}
(i) Theorem \ref{theorem32} signifies that Theorem \ref{theorem31}
also holds for $\kappa=0$, that is to say, Theorem \ref{theorem31}
holds for any $\kappa\in [0,1)$. In the proof of Theorem
\ref{theorem31}, (\ref{theorem31eq0404}) also holds for $\kappa\in
[0,1)$ since (\ref{theorem32eqa03}) is consistent with
(\ref{theorem31eq0404}) with $\kappa=0$. (ii) If $\mu(t)=-c(c>0)$
that is a UASF, then Theorem \ref{theorem31} and Theorem
\ref{theorem32} are in line with Theorem 3.1 and Theorem 3.2 in
[\ref{Yin2011}], respectively. Note that Theorem 3.1 in
[\ref{Yin2011}] does not hold for $\kappa=1$ according to Example
3.3 in [\ref{Yin2011}]. Therefore, Theorem \ref{theorem31} also does
not hold for $\kappa=1$.
\end{remark}

\begin{remark}\label{rem34}
In some existing results, $\mathcal {L}V$ must be negative definite
such as [\ref{Yin2011}, \ref{Zhao2018}] or negative semi-definite
such as [\ref{Yu2019}]. Theorem \ref{theorem31} and Theorem
\ref{theorem32} generalize the existing results and relax some
constraints on $\mathcal {L}V$ such that $\mathcal {L}V$ can be
negative or positive.
\end{remark}

\textbf{Example 1.}
For stochastic nonlinear system
\begin{align}\label{example21}
dx(t)=0.5\mu_1(t)x^{\frac{1}{3}}(t)dt-0.5x(t)dt
+x(t)\cos(x(t))dW(t),
\end{align}
where  $\mu_1(t)=2/(1+t)-|\sin2t|$ which is a UASF by [\ref{Zhou2016b}] and
$W(t)\in \mathcal{R}$ is a standard Wiener process.
Note that
\begin{align}\label{example211}
|f(t, x)|^2+|g(t, x)|^2
=&|0.5\mu_1(t)x^{\frac{1}{3}}(t)-0.5x(t)|^2+|x(t)\cos(x(t))|^2\cr
\leq &0.25\mu_1^2(t)x^{\frac{2}{3}}(t)+1.25|x(t)|^2+0.5|\mu_1(t)|x^{\frac{4}{3}}(t)\cr
\leq &x^{\frac{2}{3}}(t)+\frac{5}{4}|x(t)|^2+x^{\frac{4}{3}}(t)\cr
\leq &\frac{1}{3}|x(t)|^2+\frac{2}{3}+\frac{5}{4}|x(t)|^2+\frac{2}{3}|x(t)|^2+\frac{1}{3}\cr
\leq &H(|x(t)|^2+1),
\end{align}
where $H=9/4$, Lemma 3 in [\ref{Huang2016}] and
$\mu_1(t)=2/(1+t)-|\sin2t|\leq 2$ are used. Hence,
system (\ref{example21}) has a continuous solution by Lemma \ref{lemma22}.
Let $V_1(x)=x^2$, then we have
\begin{align}
\mathcal {L}V_1(x(t))=&2x(t)(0.5\mu_1(t)x^{\frac{1}{3}}(t)-0.5x(t))
+x^2(t)\cos^2(x(t))\cr
\leq &\mu_1(t)x^{\frac{4}{3}}(t)\cr
=&\mu_1(t)V_1^{\frac{2}{3}}(x(t)).
\end{align}
This implies that the trivial solution of system (\ref{example21}) is
finite-time stable in probability by Theorem \ref{theorem31}.

\begin{figure}[htb]
\centering
\includegraphics[height=5cm, width=15cm]{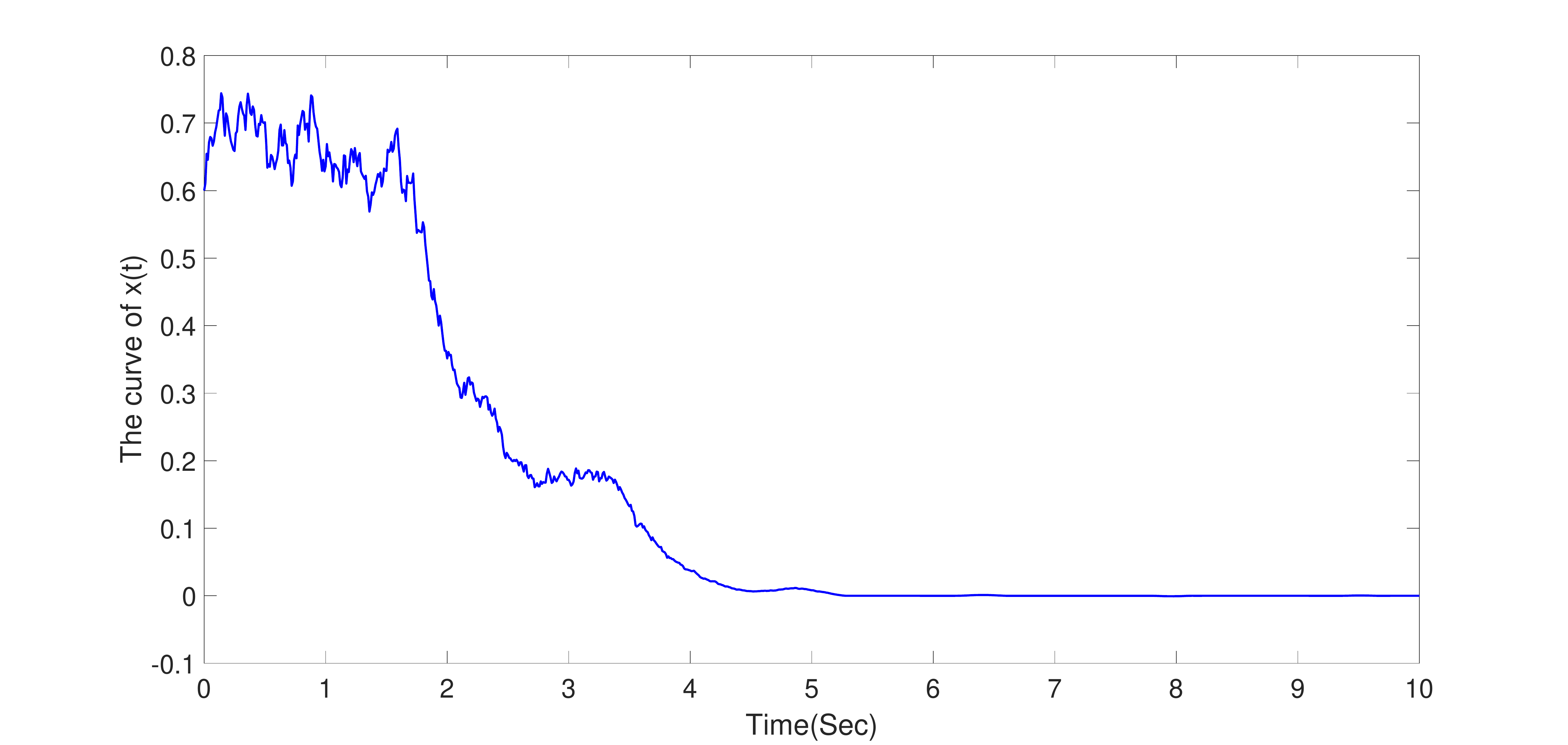}
\center{ {\small Figure 1.  The state curve of system (\ref{example21}).}}
\end{figure}

In simulation, let $x_0=0.6$, then the state curve of system (\ref{example21})
is shown in Figure 1, which shows that system
state trajectory beginning from non-zero initial
value converges to the origin in finite time.

\textbf{Example 2.}
For stochastic nonlinear system
\begin{equation}\label{example22}
\begin{cases}
dx_1(t)=f_1(t, x)dt+g_1(t, x)dW_1(t),\\
dx_2(t)=f_2(t, x)dt+g_2(t, x)dW_2(t),
\end{cases}
\end{equation}
where $f_1(t, x)=-x_1(t)+(\psi(t)-0.5)x_1^{\frac{4}{5}}(t)$,
$g_1(t, x)=\sqrt{2}x_2(t)\cos(x_1(t))$,
$f_2(t, x)=-x_2(t)+(\psi(t)-0.5)x_2^{\frac{4}{5}}(t)$,
$g_2(t, x)=\sqrt{2}x_1(t)\sin(x_2(t))$, $\psi(t)=t\sin t/(1+t)$,
$W_i(t)\in \mathcal{R}(i=1,2)$ 
are mutually independent standard Wiener processes.
System (\ref{example22}) satisfies the condition in Lemma \ref{lemma22},
so there is a continuous solution to system (\ref{example22}).

\begin{figure}[htb]
 \centering
  \includegraphics[height=5cm, width=15cm]{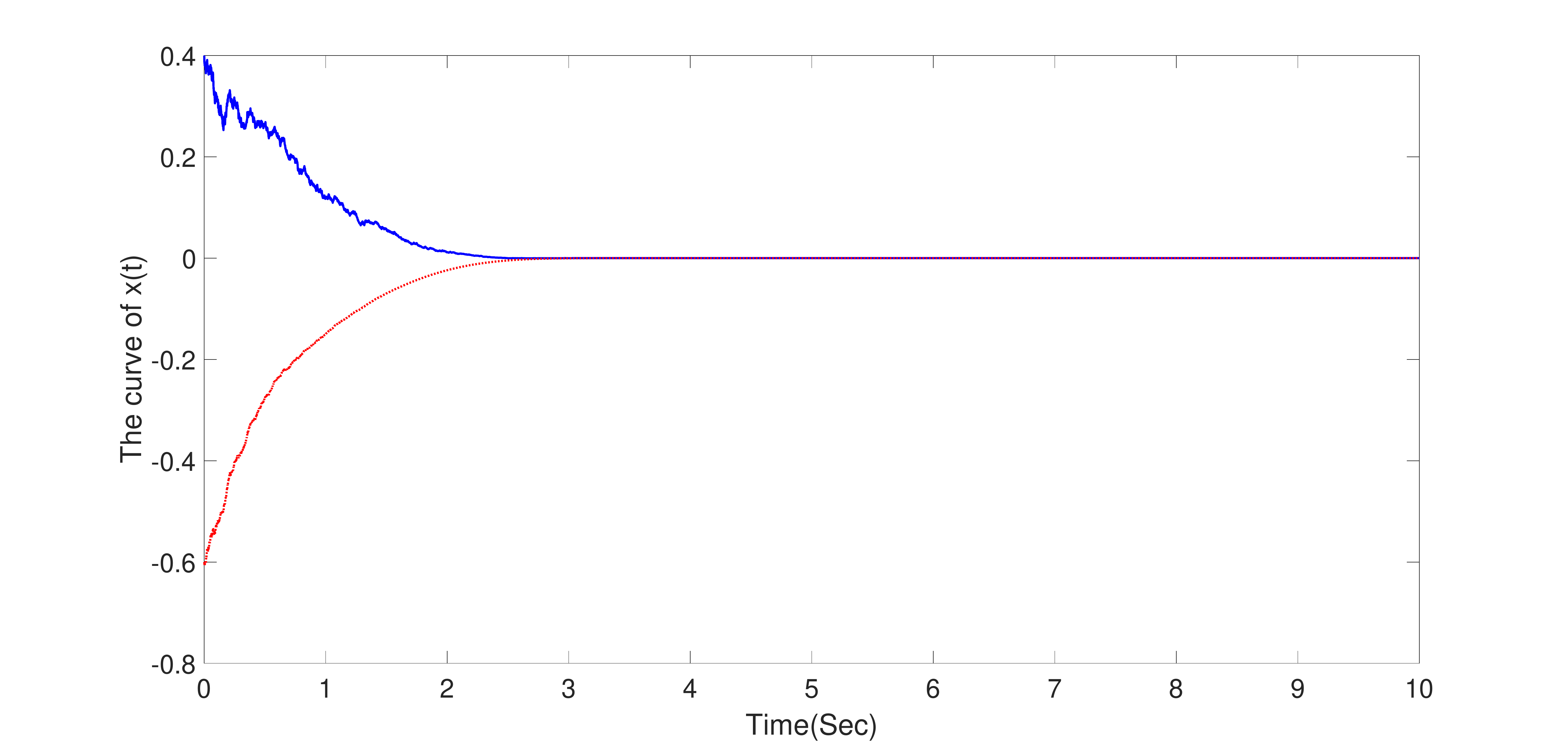}
\center{ {\small Figure 2.  The trajectory of state $x(t)$.}}
\end{figure}

Let  $V_2(x)=x_{1}^2+x_{2}^2$, then
\begin{align*}
\mathcal {L}V_2(x)\leq &2(\psi(t)-0.5)(x_{1}^{\frac{9}{5}}+x_{2}^{\frac{9}{5}})
=2\bar{\mu}_2(t)[(x_{1}^2)^{\frac{9}{10}}+(x_{2}^2)^{\frac{9}{10}}]
\leq \mu_2(t)V_2^{\frac{9}{10}}(x),
\end{align*}
where Lemma 4 in [\ref{Huang2016}] is used,
$\bar{\mu}_2(t)=\psi(t)-0.5$ and
\begin{equation*}
{\mu}_2(t)=
\begin{cases}
2^{\frac{11}{10}}\bar{\mu}_2(t) ~~&\textmd{if}~~ \psi(t)\geq 0.5\\
2\bar{\mu}_2(t)~~&\textmd{if}~~ \psi(t)<0.5\\
\end{cases}.
\end{equation*}
Since $\bar{\mu}_2(t)=\psi(t)-0.5$ is a UASF and
$\int_{t_0}^t\bar{\mu}_2(s)ds\leq 5-0.5(t-t_0)$, it is easy to
verify that ${\mu}_2(t)$ is also a UASF. Therefore, the solution of
system (\ref{example22}) with $x_0\in\mathcal{R}^2\backslash\{0\}$
is finite-time stable in probability by Theorem \ref{theorem31} with
$\kappa=9/10$. For simulation, let $x_0=(0.4, -0.6)^T$, then the
trajectories of the state $x_1(t)$ and $x_2(t)$ are described in
Figure 2. This signifies that  the trivial solution of system
(\ref{example22}) is finite-time stable in probability.

\begin{remark}\label{rem34}
(i) Since system (\ref{example21}) and (\ref{example22}) satisfy the
condition of Lemma \ref{lemma22}, system (\ref{example21}) and
(\ref{example22}) have a continuous solution, respectively. We can
also  prove that Example 1 and Example 2 satisfy Lemma
\ref{lemma23}. For example, we know that $f(x, t)$ and $g(x, t)$ in
Example 1 are locally bounded in $x$ and are   uniformly bounded  in
$t$ by (\ref{example211}). Furthermore, let
$U_1(x)=V_1(x)=x^2=|x|^2$
in Example 1, then $\mathcal {L}U_1(x(t))
=\mu_1(t)U_1^{\frac{2}{3}}(x(t))\leq 2U_1^{\frac{2}{3}}(x(t))\leq
{4}/{3}U_1+2/3 \triangleq l(t)U_1+2/3$. (ii) Example 1 and Example 2
show that the finite-time stability of stochastic nonlinear systems
can be analyzed by Theorem \ref{theorem31}. (iii) In order to
further show the validity of Theorem \ref{theorem31}, we focus on
finite-time stabilization problem for a stochastic nonlinear system,
which is given in Section V.
\end{remark}

\section{Finite-time instability theorem}\label{S4}
In this section, we discuss some sufficient conditions which ensure
that stochastic nonlinear system (\ref{eq21}) is finite-time
instable in probability. Compared with some existing stochastic
finite-time instability results (such as [\ref{Yin2011}] and
[\ref{Yu2021}]), the given finite-time instability criterion relaxes
the  condition  imposed  on $\mathcal{L}V$ and allows $\mathcal{L}V$
to be indefinite.

\begin{theorem}\label{theorem41}
Assuming that there exists a solution to system (\ref{eq21})
with any given non-zero initial value. If there exist a function
$V(t,x)\in C^{1,2}([t_0-\tau, \infty) \times \mathcal{R}^r; \mathcal{R}_+)$,
$\mathcal{K}_\infty$ functions $\underline{\gamma}$ and $\bar{\gamma}$,
and a UASF $\mu(t)$ such that
\begin{gather}
\label{theorem41eq1}
\underline{\gamma}(|x|) \leq V(t, x)\leq \bar{\gamma}(|x|),\\
\label{theorem41eq2}
\mathcal {L}V(t,x)= \mu(t)V(t, x), \\
\label{theorem41eq3}
|V_x(t,x)g(t, x)|^2\leq a(t)V^2(t,x),
\end{gather}
where  $a(t)\geq 0$ and $\int_{t_0}^\infty a(t)dt<\infty$, then the
trivial solution of  system  (\ref{eq21}) is  finite-time instable
in probability.
\end{theorem}

\begin{proof}
Firstly, let
$\tilde{V}(t,x)=e^{-\int_{t_0}^t\mu(s)ds}V(t,x)$ along system
(\ref{eq21}) and apply It\^{o}'s formula for $\tilde{V}$, then
\begin{align}\label{theorem41eq01}
d\tilde{V}=\mathcal{L}\tilde{V}dt+\tilde{V}_x^Tg(t, x(t))dW(t),
\end{align}
where $\tilde{V}_x=e^{-\int_{t_0}^t\mu(s)ds}V_x(t,x)$ and
$\mathcal{L}\tilde{V}=e^{-\int_{t_0}^t\mu(s)\,ds}(\mathcal{L}V(t,x)-\mu(t)V(t,x))= 0.$
It follows from (\ref{theorem41eq01}) that
\begin{align}\label{theorem41eq02}
e^{-\int_{t_0}^{t}\mu(s)ds}V(t, x(t))-V(t_0,x_0)
= \int_{t_0}^{t}e^{-\int_{t_0}^s\mu(v)dv}V_x^TgdW(s).
\end{align}
From (\ref{theorem41eq3}), (\ref{theorem41eq02}) and
Burkholder-Davis-Gundy inequality
(Theorem 7.3 in [Chapter 1, \ref{Mao2007}]), we can obtain that
\begin{align}\label{theorem41eq03}
&E\Big[\sup_{t_0 \leq t \leq T}
\Big(e^{-\int_{t_0}^{t}\mu(s)ds}V(t, x(t))\Big)^2\Big]\cr
\leq &8E\int_{t_0}^{T}a(t)\Big(e^{-\int_{t_0}^t\mu(v)dv}V(t,x(t))\Big)^2dt
+2V^2(t_0, x_0)\cr
\leq &8\int_{t_0}^{T}a(t)E\Big[\sup_{t_0 \leq s \leq T}
\Big(e^{-\int_{t_0}^s\mu(v)dv}V(s,x(s))\Big)^2\Big]dt
+2V^2(t_0, x_0).
\end{align}
Further, by Gronwall's inequality
(Theorem 8.1 in [Chapter 1, \ref{Mao2007}]), it can
be derived from (\ref{theorem41eq03}) that
\begin{align}\label{theorem41eq04}
E\Big[\sup_{t_0 \leq t \leq T}
\Big(e^{-\int_{t_0}^{t}\mu(s)ds}V(t, x(t))\Big)^2\Big]
\leq 2e^{\int_{t_0}^{\infty}8a(t)dt}V^2(t_0, x_0)
\leq 2e^{\int_{t_0}^{\infty}8a(t)dt}\bar{\gamma}^2(|x_0|)\triangleq H_0.
\end{align}
Let $T\rightarrow \infty$, then by Fatou's lemma, we have
$$E\Big[\sup_{t\in[t_0,\infty)}
\Big(e^{-\int_{t_0}^{t}\mu(s)ds}V(t, x(t))\Big)^2\Big]
\leq H_0<\infty,$$
which means that
\begin{align}\label{theorem41eq05}
E\Big[\sup_{t\in[t_0,\infty)}
\Big(e^{-\int_{t_0}^{t}\mu(s)ds}V(t, x(t))\Big)\Big]<\infty.
\end{align}

For any given $x_0\in \mathcal{R}^r\backslash\{0\}$,
let $\sigma_n=\inf\{t\geq t_0:|x(t)|>n\}$. It follows from
(\ref{theorem41eq02}) that
\begin{align}\label{theorem41eq06}
E\Big\{\Big[e^{-\int_{t_0}^{t}\mu(s)ds}
V(t, x(t))\Big]\Big|_{t=\sigma_n \wedge \sigma_b}\Big\}
= V(t_0,x(t_0)),
\end{align}
where $\sigma_b$ represents any bounded stopping time. Moreover,
\begin{align}\label{theorem41eq07}
0 \leq e^{-\int_{t_0}^{\sigma_n \wedge \sigma_b}\mu(s)ds}
V(\sigma_n \wedge \sigma_b, x(\sigma_n \wedge \sigma_b))
\leq \sup_{t\in[t_0,\infty)}\Big(e^{-\int_{t_0}^{t}\mu(s)ds}V(t,x(t))\Big).
\end{align}
From (\ref{theorem41eq05}) and (\ref{theorem41eq07}),
apply Lebesgue's dominated convergence theorem for (\ref{theorem41eq06}),
then we  have
\begin{align}\label{theorem41eq08}
E\Big[e^{-\int_{t_0}^{\sigma_b}\mu(s)ds}
V(\sigma_b, x(\sigma_b))\Big]=V(t_0,x(t_0)),
\end{align}
where $\sigma_n\rightarrow\infty$ as $n\rightarrow\infty$ is used.
This signifies that $e^{-\int_{t_0}^{t}\mu(s)ds}V(t, x(t))$ is a uniformly integrable
martingale. Hence, by  martingale convergence theorem
(Theorem 7.11 in [\ref{Klebaner2005}]), we have
\begin{align}\label{theorem41eq09}
0 \leq \lim_{t\rightarrow\infty} e^{-\int_{t_0}^{t}\mu(s)ds}
V(t, x(t))=\vartheta(\omega)<\infty, ~~~a.s.,
\end{align}
where $\vartheta(\omega)$ is a random variable.

Secondly, the conditions (\ref{theorem41eq1}) and (\ref{theorem41eq2}) ensure that
there is a stable solution to system (\ref{eq21}) in probability.
In fact, for every given $x_0\in \mathcal{R}^r\backslash\{0\}$,
let $T^*=\max\{T\geq t_0: \int_{t_0}^{T}\mu(t)dt= 0\}$.
If $t\in[t_0, T^*]$,  from (\ref{theorem41eq1}), (\ref{theorem41eq02})
and Remark \ref{rem21}, we obtain that
\begin{align}\label{theorem41eq10}
E[\underline{\gamma}(|x(t)|)]\leq EV(t, x(t))\leq \bar{\gamma}(|x_0|)e^{d_\mu}.
\end{align}
If $t\in(T^*, \infty)$, by (\ref{theorem41eq02}), we have
\begin{align}\label{theorem41eq101}
E[\underline{\gamma}(|x(t)|)]\leq &EV(t, x(t))
=V(t_0, x_0)e^{\int_{t_0}^{t}\mu(s)ds}
\leq \bar{\gamma}(|x_0|),
\end{align}
where $\int_{t_0}^{t}\mu(s)ds=\int_{t_0}^{T^*}\mu(s)ds+\int_{T^*}^{t}\mu(s)ds
=\int_{T^*}^{t}\mu(s)ds<0$ is used.

Hence, for any $t\geq t_0$, we have
\begin{align*}
E[\underline{\gamma}(|x(t)|)]\leq \bar{\gamma}(|x_0|)e^{d_\mu}.
\end{align*}
From Chebyshev's inequality,
for any $\varepsilon\in (0,1)$ and $R>0$, we have
\begin{align}\label{theorem41eq11}
P\{|x(t)|\geq R, t\geq t_0\}
= P\{\underline{\gamma}(|x(t)|)
\geq \underline{\gamma}(R), t\geq t_0\}
\leq \underline{\gamma}^{-1}(R)E[\underline{\gamma}(|x(t)|)]
\leq \underline{\gamma}^{-1}(R)\bar{\gamma}(|x_0|)e^{d_\mu}
\leq \varepsilon
\end{align}
as $|x_0|<\tilde{\delta}(\varepsilon, R)
=\bar{\gamma}^{-1}(\varepsilon\underline{\gamma}(R)e^{-d_\mu})$.
This implies that $P\{|x(t)|< R, t\geq t_0\}\geq 1-\varepsilon$ as
$|x_0|<\tilde{\delta}(\varepsilon, R)$.

Finally, consider (\ref{theorem41eq09}),
we can prove that  it is not finite-time attractive in probability
for the  solution of system (\ref{eq21})
by the same method as Theorem 4.1 in [\ref{Yin2011}].
\end{proof}

\begin{remark}\label{rem41}
In Theorem 4.1 of [\ref{Yin2011}], $\mathcal {L}V(t,x)= -c_3V(t, x)$
with $c_3>0$ is used. It is clear that $\mu(t)=-c_3(c_3>0)$ is a
UASF. In Theorem \ref{theorem41}, the constraint condition of
$\mathcal {L}V(t,x)= -c_3V(t, x)$ is replaced by $\mathcal
{L}V(t,x)=\mu(t) V(t, x)$. Hence, Theorem \ref{theorem41}
 improves  Theorem 4.1 of [\ref{Yin2011}].
\end{remark}

\section{A simulation example}\label{S5}

\textbf{Example 3.} Consider the  stochastic nonlinear system
with stochastic inverse dynamics $(\Sigma)$
\begin{equation}\label{eq51}
\begin{cases}
d\chi(t)=\varphi(t)\chi^{\beta_1}(t)dt+\cos (x_1(t))\chi^{\beta_2}(t)dB_0(t),\\
dx_1(t)=x_2(t)dt\\
dx_2(t)=u(t)dt+x_2^{\beta_3}(t)\sin(\chi(t))dB(t),
\end{cases}
\end{equation}
where $\beta_1=(2l-1)/(2l+1)\in(0.5, 1)$,
$\beta_2=2l/(2l+1)\in(0.5, 1)$ and
$\beta_3=(2l-2)/(2l-1)\in(0.5, 1)$ with
$l\in\mathcal {N}$ and $l\geq 2$.
$B_0(t)\in \mathcal{R}$ and $B(t)\in \mathcal{R}$
are mutually independent standard Brownian motions(Wiener processes).
$\varphi(t)=0.5({t\cos t}/{(1+t)}-1.5)$.

For $\chi$-subsystem, we select $V_0=\chi^2$, then
$\mathcal{L}V_0=2\varphi(t)\chi^{\beta_1+1}+\chi^{2\beta_2}\cos^2 (x_1)
\leq (2\varphi(t)+1)\chi^{2\beta_2}
=\mu_0(t)V_0^{\beta_2},$
where $\beta_1+1=2\beta_2$ and $\mu_0(t)=2\varphi(t)+1$ is a
UASF since $\int_{t_0}^t\mu_0(s)ds\leq -0.5(t-t_0)+5$.

For system $\Sigma$, we introduce $V_1=V_0+W_0$ with $W_0=0.5z_1^2$ and $z_1=x_1$,
then
\begin{align}\label{eq52}
\mathcal {L}V_1=&\mathcal {L}V_0+\mathcal {L}W_0
=\mathcal {L}V_0+z_1x_2
\leq \mu_0(t)V_0^{\beta_2}+z_1(x_2-\alpha)+z_1\alpha.
\end{align}
Let the stabilizing function $\alpha=-c_1z_1^{\lambda}$ with $\lambda=\beta_1$ and
$c_1>0$ being a design constant, then (\ref{eq52}) can be changed
into
\begin{align}\label{eq53}
\mathcal
{L}V_1\leq\mu_0(t)V_0^{\beta_2}+z_1(x_2-\alpha)-c_1z_1^{1+\lambda}.
\end{align}
Further, we introduce
$V_2=V_1+W_1, W_1=\int_{\alpha}^{x_2}(v^{\frac{1}{\lambda}}
-\alpha^{\frac{1}{\lambda}})^{2-\lambda}dv,$
and $z_2=x_2^{\frac{1}{\lambda}}-\alpha^{\frac{1}{\lambda}}$,
then it follows from Proposition B.1-B.2 in [\ref{Qian2001b}] that
$V_2$ is a positive definite function and
\begin{align}\label{eq533}
W_0+W_1\leq 2(z_{1}^2+z_{2}^2).
\end{align}
Meanwhile, we can deduce that
\begin{align}\label{eq54}
\mathcal {L}V_2
=\mathcal {L}V_1+\mathcal {L}W_1
=\mathcal {L}V_1+\frac{\partial W_1}{\partial x_1}x_2
+\frac{\partial W_1}{\partial x_2}u
+\frac{1}{2}\frac{\partial^2 W_1}{\partial x_2^2}x_2^{2\beta_3}\sin^2(\chi).
\end{align}
Note that
\begin{align}\label{eq55}
\frac{\partial W_1}{\partial x_1}x_2
=&-(2-\lambda)\frac{\partial \alpha^{\frac{1}{\lambda}}}{\partial x_1}x_2
\int_{\alpha}^{x_2}(v^{\frac{1}{\lambda}}-\alpha^{\frac{1}{\lambda}})^{1-\lambda}dv\cr
=&c_1^{\frac{1}{\lambda}}(2-\lambda)x_2
\int_{\alpha}^{x_2}(v^{\frac{1}{\lambda}}-\alpha^{\frac{1}{\lambda}})^{1-\lambda}dv\cr
\leq &c_1^{\frac{1}{\lambda}}(2-\lambda)|z_2|^{1-\lambda}|x_2-\alpha||x_2|\cr
\leq &c_1^{\frac{1}{\lambda}}(2-\lambda)2^{1-\lambda}|z_2|^{1-\lambda}|z_2|^\lambda|x_2|\cr
\leq &c_1^{\frac{1}{\lambda}}(2-\lambda)2^{1-\lambda}|z_2|(|x_2-\alpha|+|\alpha|)\cr
\leq &c_1^{\frac{1}{\lambda}}(2-\lambda)2^{1-\lambda}(2^{1-\lambda}z_2^{1+\lambda}
+c_1|z_2||z_1|^\lambda)\cr
\leq &c_1^{\frac{1}{\lambda}}(2-\lambda)2^{1-\lambda}\Big(2^{1-\lambda}z_2^{1+\lambda}
+\frac{c_1h_1^{-\lambda}}{1+\lambda}z_2^{1+\lambda}
+\frac{c_1h_1\lambda}{1+\lambda}z_1^{1+\lambda}\Big)\cr
=&c_1^{\frac{1}{\lambda}}(2-\lambda)[2^{2(1-\lambda)}z_2^{1+\lambda}
+\frac{2^{1-\lambda}c_1
h_1\lambda}{1+\lambda}
z_1^{1+\lambda}]
+\frac{2^{1-\lambda}c_1^{1+\frac{1}{\lambda}}h_1^{-\lambda}}{1+\lambda}(2-\lambda)z_2^{1+\lambda},
\end{align}
\begin{align}\label{eq57}
\frac{1}{2}\frac{\partial^2 W_1}{\partial x_2^2}x_2^{2\beta_3}\sin^2(\chi)
=&\frac{2-\lambda}{2}z_2^{1-\lambda}\frac{1}{\lambda}x_2^{\frac{1}{\lambda}-1}x_2^{2\beta_3}\sin^2(\chi)\cr
\leq &\frac{2-\lambda}{2\lambda}z_2^{1-\lambda}x_2^{2}\cr
=&\frac{2-\lambda}{2\lambda}z_2^{1-\lambda}(x_2^{\frac{1}{\lambda}})^{2\lambda}\cr
\leq &\frac{2-\lambda}{2\lambda}z_2^{1-\lambda}2^{2\lambda-1}
(z_2^{2\lambda}+\alpha^{2})\cr
\leq &\frac{2-\lambda}{\lambda}z_2^{1-\lambda}
(z_2^{2\lambda}+\alpha^{2})\cr
=&\frac{2-\lambda}{\lambda}z_2^{1+\lambda}
+\frac{2-\lambda}{\lambda}c_1^2z_2^{1-\lambda}z_1^{2\lambda}\cr
\leq &\frac{2-\lambda}{\lambda}z_2^{1+\lambda}
+\frac{2-\lambda}{\lambda}c_1^2
\Big(\frac{2\lambda h_2}{1+\lambda}z_1^{1+\lambda}
+\frac{1-\lambda}{1+\lambda}h_2^{-\frac{2\lambda}{1-\lambda}}z_2^{1+\lambda}\Big)\cr
=& \frac{2-\lambda}{\lambda}z_2^{1+\lambda}
+\frac{2-\lambda}{\lambda}\frac{1-\lambda}{1+\lambda}c_1^2h_2^{-\frac{2\lambda}{1-\lambda}}
z_2^{1+\lambda}
+\frac{2(2-\lambda)}{1+\lambda}c_1^2h_2z_1^{1+\lambda},
\end{align}
\begin{align}\label{eq56}
\frac{\partial W_1}{\partial x_2}u=z_2^{2-\lambda}u,
\end{align}
where $1-\lambda=2/(2l+1)$ and $1+\lambda=4l/(2l+1)$,
$h_1>0$ and $h_2>0$ are to  be  designed constants and Lemma 4 in
[\ref{Huang2016}] is used.

Substituting (\ref{eq53}), (\ref{eq55})-(\ref{eq56}) into (\ref{eq54}), we can obtain that
\begin{align}\label{eq58}
\mathcal {L}V_2
\leq \mu_0(t)V_0^{\beta_2}+z_1(x_2-\alpha)-c_1z_1^{1+\lambda}
+\tilde{d}_2z_2^{1+\lambda}+\tilde{d}_1z_1^{1+\lambda}
+z_2^{2-\lambda}u,
\end{align}
where
$\tilde{d}_2=
(2-\lambda){2^{1-\lambda}c_1^{1+\frac{1}{\lambda}}h_1^{-\lambda}}/{(1+\lambda)}
+(2-\lambda)c_1^{\frac{1}{\lambda}}2^{2(1-\lambda)}
+{(1-\lambda)(2-\lambda)c_1^2h_2^{-\frac{2\lambda}{1-\lambda}}}/{\lambda(1+\lambda)}
+({2-\lambda})/{\lambda}$,
$\tilde{d}_1={2^{1-\lambda}(2-\lambda)c_1^{1+\frac{1}{\lambda}}h_1\lambda}/{(1+\lambda)}
+{2(2-\lambda)c_1^2h_2}/{(1+\lambda)}$.
In addition,
\begin{align}\label{eq59}
z_1(x_2-\alpha)
\leq |z_1||(x_2^{\frac{1}{\lambda}})^{\lambda}-(\alpha^{\frac{1}{\lambda}})^{\lambda}|
\leq 2^{1-\lambda}|z_1||z_2|^{\lambda}
\leq \frac{2^{1-\lambda}}{1+\lambda}[h_3z_1^{1+\lambda}
+\lambda h_3^{-\frac{1}{\lambda}}z_2^{1+\lambda}],
\end{align}
where $h_3>0$ is a  constant that will be designed,  and Lemma 2.3
in [\ref{Huang2005}] is applied. Hence, (\ref{eq58}) and
(\ref{eq59}) mean  that
\begin{align}\label{eq510}
\mathcal {L}V_2\leq
\mu_0(t)V_0^{\beta_2}-c_1z_1^{1+\lambda}
+{d}_2z_2^{1+\lambda}
+{d}_1z_1^{1+\lambda}
+z_2^{2-\lambda}u,
\end{align}
where $d_1=\tilde{d}_1+{2^{1-\lambda}h_3}/{(1+\lambda)}$ and
$d_2=\tilde{d}_2+{2^{1-\lambda}\lambda h_3^{-\frac{1}{\lambda}}}/{(1+\lambda)}$.

Let $h_1=(1+\lambda)2^{\lambda-1}c_1^{-\frac{1}{\lambda}}/6\lambda(2-\lambda)$,
$h_2=(1+\lambda)c_1^{-1}/12(2-\lambda)$
and $h_3=2^{\lambda}c_1(1+\lambda)/12$, then $d_1=c_1/2$ and (\ref{eq510}) can be rewritten as
\begin{align}\label{eq511}
\mathcal {L}V_2\leq
\mu_0(t)V_0^{\beta_2}-0.5{c_1}z_1^{1+\lambda}
+{d}_2z_2^{1+\lambda}+z_2^{2-\lambda}u.
\end{align}
Further, we choose the controller
\begin{align}\label{eq512}
u=-{d}_2z_2^{2\lambda-1}-0.5{c_2}z_2^{2\lambda-1},
\end{align}
where $c_2>0$ is a design constant.
Then, together with $1+\lambda=1+\beta_1=2\beta_2$,
(\ref{eq511}) and  (\ref{eq512}) lead to
\begin{align}\label{eq513}
\mathcal {L}V_2
\leq &\mu_0(t)V_0^{\beta_2}-0.5{c_1}(z_1^2)^{\beta_2}-0.5{c_2}(z_2^2)^{\beta_2}\cr
\leq &\mu_0(t)V_0^{\beta_2}-0.5{c_0}((z_1^2)^{\beta_2}+(z_2^2)^{\beta_2})\cr
\leq &\mu_0(t)V_0^{\beta_2}-0.5{c_0}(z_1^2+z_2^2)^{\beta_2}\cr
\leq &\mu_0(t)V_0^{\beta_2}-0.25{c_0}(W_0+W_1)^{\beta_2}\cr
\leq &\tilde{\mu}(t)(V_0^{\beta_2}+(W_0+W_1)^{\beta_2})\cr
\leq &{\mu}(t)V_2^{\beta_2},
\end{align}
where
\begin{equation*}
\mu(t)=
\begin{cases}
2^{1-\beta_2}\tilde{\mu}(t)~~&\textmd{if}~~ \tilde{\mu}(t)\geq 0\\
\tilde{\mu}(t)~~&\textmd{if}~~ \tilde{\mu}(t)< 0\\
\end{cases},~~~~
\tilde{\mu}(t)=
\begin{cases}
\mu_0(t)~~&\textmd{if}~~ \mu_0(t)\geq -0.25c_0\\
-0.25c_0~~&\textmd{if}~~ \mu_0(t)<-0.25c_0\\
\end{cases},
\end{equation*}
$c_0=\min\{c_1, c_2\}$.
Moreover, Lemma 4 in [\ref{Huang2016}], (\ref{eq533})
and $0.5^{\beta_2}>0.5$  are used in (\ref{eq513}).
Note that $\mu_0(t)$ is a UASF, so $\tilde{\mu}(t)$ and $\mu(t)$
are UASFs by Remark \ref{rem21}.

Therefore,  let $\xi=(\chi, z_1, z_2)^T$ and Lyapunov function
$V=V_0+W_0+W_1$ for system $\Sigma$, then the positive function
$V=V_0+W_0+W_1\leq \chi^2+2(z_1^2+z_2^2)\leq 2|\xi|^2$ and $\mathcal
{L}V\leq {\mu}(t)V^{\beta_2}$, where  (\ref{eq533}) is considered.
By  Theorem \ref{theorem31}, stochastic nonlinear system
(\ref{eq51})  is finite-time stabilizable in probability with the
controller (\ref{eq512}).

\begin{figure}[htb]
\centering
\includegraphics[height=10cm, width=15cm]{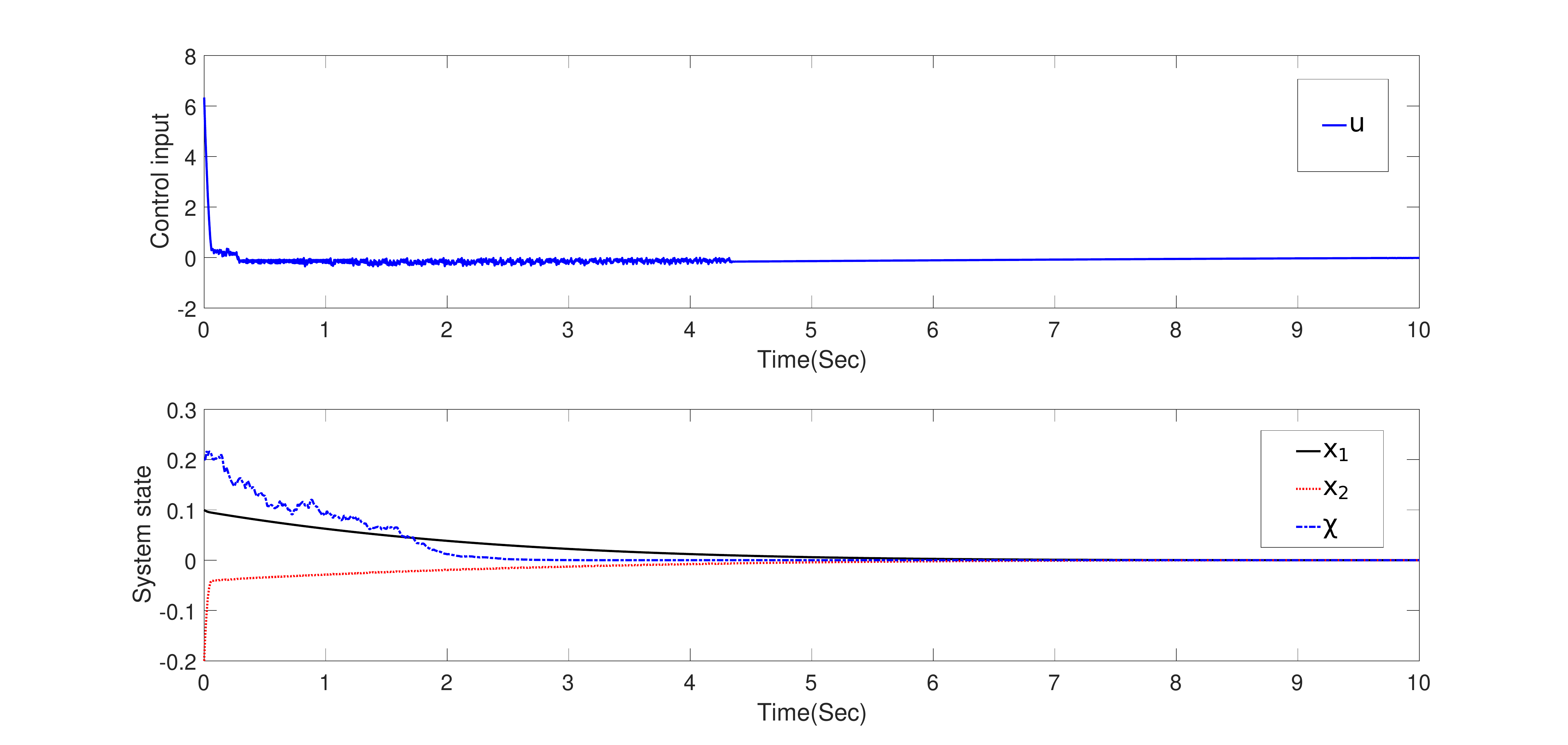}
\center{ {\small Figure 3. The state curves of  system (\ref{eq51}) with (\ref{eq512}).}}
\end{figure}
In simulation, we choose design parameter  $l=4$, $c_1=0.3$ and $c_2=0.3$, system
initial value $\chi(0)=0.2$, $x_1(0)=0.1$ and $x_2(0)=-0.2$,
then system control input trajectory and system state curves are given in Figure 3.

\section{Conclusions}\label{S7}
This paper has further studied  the finite-time stability and
instability in probability for stochastic nonlinear systems. A
weaker sufficient condition, which ensures that the considered
system has a global solution,  has been  presented. Some improved
finite-time stability and instability theorems have been  given by
the UASF. The obtained stability and instability results relax the
strict constraint conditions  on  $\mathcal {L}V$ existing in
previous references. Some examples are given to illustrate that the
obtained  stability theorems can be used to analysis  and synthesis
of  stochastic nonlinear systems  including autonomous systems and
forced systems.

\section*{ References}
{\small\setlength{\baselineskip}{13pt}
\newcounter{a}
\noindent
\begin{list}{\parbox[t]{0.6cm}{\arabic{a}\hfill ${}$}}{\usecounter{a}}
\setlength{\itemsep}{0cm}

\item\label{Khalil1992}
Khalil, H. K. (2002). {\em Nonlinear Systems} (3rd. ed.). New
Jersey: Prentice Hall.

\item\label{Khasminskii1980}
Khasminski, R.  (2012). {\em Stochastic Stability of Differential
Equations}(2nd. ed.).  New York: Springer.

\item\label{Mao2007}
Mao, X. R. (2007).
{\em Stochastic Differential Equations and
Applications}(2nd ed.). Chichester: Horwood Publishing.

\item\label{yaoliqiang} Yao, L., Zhang,  W., \& Xie, X. J. (2020).
Stability analysis of random nonlinear systems with time-varying
delay and its application. {\em  Automatica}, 117, Article 108994.

\item\label{xuelingrong}   Xue, L.,  Zhang, T.,  Zhang, W.,  \& Xie, X. J. (2018). Global
adaptive stabilization and tracking control for high-order
stochastic nonlinear systems with time-varying delay. {\em IEEE
Transactions on  Automatic Control}, 63(9), 2928-2943.

\item\label{Haimo1986}
Haimo, V. T. (1986). Finite time controllers. {\em SIAM Journal on
Control and Optimization}. 24(4), 760-770.

\item\label{Bhat2000}
Bhat, S. P., \&  Bernstein, D. S. (2000). Finite-time stability of
continuous autonomous systems. {\em SIAM Journal on Control and
Optimization}. 38(3), 751-766,

\item\label{Moulay2008}
Moulay, E., \& Perruquetti, W. (2008). Finite time stability
conditions for non-autonomous continuous systems. {\em International
Journal of Control}. 81(5), 797-803.

\item\label{Nersesov2008}
Nersesov, S. G., Haddad, W. M., \& Hui, Q. (2008). Finite time
stabilization of nonlinear dynamical systems via control vector
Lyapunov functions. {\em Journal of The Franklin Institute}, 345(7),
819-837.

\item\label{Yu2005}
Yu, S. H.,  Yu, X. H., Shirinzadeh, B., \& Man, Z. H. (2005).
Continuous finite-time control for robotic manipulators
with terminal sliding mode.
{\em Automatica}, 41(11), 1957-1964.

\item\label{Banavar2006}
Banavar, R. N., \& Sankaranarayanan, V. (2006).
{\em Switched Finite Time Control of a Class of  Underactuated Systems}.
Berlin Heidelberg: Springer-Verlag.

\item\label{Wang2013}
Wang, Y. Z., \& Feng, G. (2013). On finite-time stability and
stabilization of nonlinear port-controlled Hamiltonian systems. {\em
Science China Information Sciences}, 56(10), 251-264.

\item\label{Wu2017}
Wu, J.,  Li, J., Zong, G., \&  Chen, W. (2017).
Global finite-time adaptive stabilization of nonlinearly
parameterized systems with multiple unknown control directions.
{\em IEEE Transactions on Systems, Man, and Cybernetics: Systems}, 47(7), 298-307.

\item\label{Zhai2018}
Zhai, J. Y., Song, Z. B., Fei, S., \& Zhu, Z. (2018).
Global finite-time output feedback stabilisation for
a class of switched high-order nonlinear systems.
{\em International Journal of Control}, 91(1), 170-180.

\item\label{Du2018}
Du, H.,  Zhai, J. Y., Chen, M. Z. Q., \&  Zhu, W. (2019).
Robustness analysis of a continuous higher-order
finite-time control system under sampled-data control.
{\em IEEE Transactions on Automatic Control}, 64(6), 2488-2494.

\item\label{Ming2021}
Li, Y. M., Hu, J.,  Yang T., \& Fan, Y. (2021).
Global finite-time stabilization of switched high-order
rational power nonlinear systems.
{\em Nonlinear Analysis: Hybrid Systems}, 40, Article 101007.

\item\label{Yin2011}
Yin, J. L., Khoo, S., Man, Z. H., \& Yu, X. H. (2011).
Finite-time stability and instability of stochastic nonlinear systems.
{\em Automatica}, 47(12), 2671-2677.

\item\label{Yu2019}
Yu, X.,  Yin, J. L.,  \& Khoo, S. (2019).
Generalized Lyapunov criteria on
finite-time stability of stochastic nonlinear systems.
{\em Automatica}, 107, 183-189.

\item\label{Yu2021}
Yu, X., Yin,  J. L., \& Khoo, S. (2021).
New Lyapunov conditions of stochastic finite-time stability
and instability of nonlinear time-varying SDEs.
{\em International Journal of Control}, 94(6), 1674-1681.

\item\label{Wang2015}
Wang, H., \& Zhu, Q. X. (2015).
Finite-time stabilization of high-order
stochastic nonlinear systems in strict-feedback form.
{\em Automatica}, 54, 284-291.

\item\label{xiexuejun1} Xie, X. J.,  $\&$  Li, G. J. (2019).
Finite-time output-feedback stabilization of high-order
nonholonomic systems. {\em International Journal of Robust and
Nonlinear Control},  29(9),  2695-2711.

\item\label{Min2017}
Min, H. F., Xu, S. Y., Ma, Q., Qi, Z., \& Zhang, Z. Q. (2017).
Finite-time stabilization of stochastic nonlinear systems with SiISS inverse dynamics.
{\em International Journal of Robust Nonlinear Control}, 27(18), 4648-4663.

\item\label{Zhao2018}
Zhao, G. H., Li, J. C., \& Liu, S. J. (2018).
Finite-time stabilization of weak solutions for a class
of non-local Lipschitzian stochastic nonlinear systems
with inverse dynamics.
{\em Automatica}, 98, 285-295.

\item\label{Huang2016}
Huang, S. P., \& Xiang, Z. R. (2016).
Finite-time stabilization of switched stochastic nonlinear
systems with mixed odd and even powers.
{\em Automatica}, 73, 130-137.

\item\label{Song2019}
Song, Z. B., \& Zhai, J. Y. (2019).
Global finite-time stabilization for switched stochastic
nonlinear systems via output feedback.
{\em Journal of The Franklin Institute}, 365(3), 1379-1395.

\item\label{Zhou2018}
Zhou, B.,  \&  Luo, W. W.  (2018).
Improved Razumikhin and Krasovskii stability
criteria for time-varying stochastic time-delay systems.
{\em Automatica}, 89, 382-391.

\item\label{Zhou2016}
Zhou, B.,  \& Egorov,  A. V. (2016a).
Razumikhin and Krasovskii stability
theorems for time-varying time-delay systems.
{\em Automatica}, 71, 281-291.

\item\label{Situ2005}
Situ, R. (2005).
{\em Theory of Stochastic Differential Equations with Jumps
and Applications: Mathematical and Analytical Techniques with
Applications to Engineering}.  New York: Springer.

\item\label{Skorokhod1965}
Skorokhod, A. V. (1965).
{\em Studies in the theory of random processes}.
Massachusetts: Addison-wesley.

\item\label{Zhao2015}
Zhao, P.,  Feng, W.,  Zhao, Y., \& Kang, Y.  (2015).
Finite-time stochastic input-to-state stability of
switched stochastic nonlinear systems.
{\em Applied Mathematics and Computation}, 268, 1038-1054.

\item\label{Zhou2016b}
Zhou, B. (2016b).
On asymptotic stability of linear time-varying systems.
{\em Automatica}, 68, 266-276.

\item\label{Klebaner2005}
Klebaner, F. C. (2005).
{\em Introduction to Stochastic Calculus with
Applications}(2nd ed.). London: Imperial College Press.

\item\label{Qian2001b}
 Qian C. J., \& Lin, W. (2001).
A continuous feedback approach to global strong stabilization of nonlinear systems.
{\em IEEE Transactions on Automatic Control},  46(7), 1061-1079.

\item\label{Huang2005}
Huang, X. Q.,  Lin, W.,  \& Yang, B. (2005).
Global finite-time stabilization of a class of uncertain nonlinear systems.
{\em Automatica}, 42(5),  881-888.

\end{list}}

\end{document}